\documentclass[twoside,11pt,fleqn]{article}
\usepackage{latexsym}
\usepackage{amsmath, amsthm}
\usepackage{graphicx} 
\usepackage{amsfonts} 
\usepackage{amssymb}

\DeclareMathOperator{\oinv}{oinv}
\DeclareMathOperator{\onsp}{onsp}
\DeclareMathOperator{\negg}{neg}
\DeclareMathOperator{\oneg}{oneg}
\DeclareMathOperator{\inv}{inv}
\DeclareMathOperator{\nsp}{nsp}

\theoremstyle{plain}
\newtheorem{thm}{Theorem}[section]
\newtheorem{pro}[thm]{Proposition}
\newtheorem{lem}[thm]{Lemma}
\newtheorem{con}[thm]{Conjecture}
\newtheorem{cor}[thm]{Corollary}

\theoremstyle{definition}
\newtheorem{defn}[thm]{Definition}

\theoremstyle{remark}

\newcommand{\mc}{{\footnotesize\left[ \begin{array}{c}  m \\
\left\lfloor \frac{|I_{1}|+1}{2}\right\rfloor,\ldots,\left\lfloor \frac{|I_{s}|+1}{2}\right\rfloor \end{array} \right]_{x^2}}}
\newcommand{\mcb}{\footnotesize{\left[ \begin{array}{c} m \\
\left\lfloor \frac{|J_{1}|+1}{2}\right\rfloor,\ldots,\left\lfloor \frac{|J_{s}|+1}{2}\right\rfloor\end{array} \right]_{x^2}}}

\newcommand{\p}{\noindent}

\newcommand{\qb}{B_n^{J}}

\newcommand{\sumb}{\sum_{\sigma \in {\qb}} (-1)^{\ell_B(\sigma)} x^{L_B(\sigma)}}
\newcommand{\N}{\mathbb N}

\newcommand{\f}{(-1)^{\ell(\sigma)}x^{L(\sigma)}}

\newcommand{\eqdef}{:=}

\begin{document}

\begin{center}

{\Large \bf Odd length for even hyperoctahedral groups and\\
 signed generating functions \footnote{2010 Mathematics Subject
Classification: Primary 05A15; Secondary 05E15, 20F55.}}
 \vspace{0.8cm} 

 Francesco Brenti \\

Dipartimento di Matematica \\

Universit\`{a} di Roma ``Tor Vergata''\\

Via della Ricerca Scientifica, 1 \\

00133 Roma, Italy \\

{\em brenti@mat.uniroma2.it } \\

 \vspace{0.5cm}

 Angela Carnevale \footnote{Partially supported by German-Israeli Foundation for Scientific Research and Development, grant no. 1246.} \\

Fakultat f\"ur Mathematik \\

Universit\"at Bielefeld\\

D-33501 Bielefeld, Germany \\

{\em acarneva1@math.uni-bielefeld.de } \\
\end{center}

\vspace{1cm}

\begin{abstract}
We define a new statistic on the even hyperoctahedral
groups which is a natural analogue of the odd length statistic recently
defined and studied on Coxeter groups of types $A$ and $B$. We compute  the signed  (by length)
generating  function of this statistic over the whole group and over 
its maximal and some other quotients and show that it always factors nicely.
We also present some conjectures.
\end{abstract}

\section{Introduction}
The signed (by length) enumeration of the symmetric group, and other finite Coxeter groups by
various statistics is an active area of research (see, e.g., \cite{AGR, Bia, BC, Cas, DF, L, Man, Mon, Rei, Siv, Wa}). For example, the
signed enumeration of classical Weyl groups by major index was carried out by 
Gessel-Simion in \cite{Wa} (type $A$), by Adin-Gessel-Roichman in \cite{AGR} (type $B$)
and by Biagioli in \cite{Bia} (type $D$), that by descent by Desarmenian-Foata in \cite{DF}
(type $A$) and by Reiner in \cite{Rei} (types $B$ and $D$), while that by excedance by Mantaci
in \cite{Man} and independently by Sivasubramanian in \cite{Siv} (type $A$) and by Mongelli in \cite{Mon}
(other types).

In \cite{KV}, \cite{VS} and \cite{VS2} two statistics were introduced on the symmetric and hyperoctahedral
groups, in connection with the enumeration of partial flags in a quadratic space and the
study of local factors of representation zeta functions of certain groups, respectively
(see \cite{KV} and \cite{VS2}, for details). These statistics combine combinatorial and
parity conditions and have been called the ``odd length'' of the respective groups.
In \cite{KV} and \cite{VS2} it was conjectured that the signed (by length) generating functions
of these statistics over all the quotients of the corresponding groups always factor in a 
very nice way, and this was proved in \cite{BC} (see also \cite{CarT}) for types $A$ and $B$
and independently, and in a different way, in \cite{L} for type $B$.

In this paper we define a natural analogue of these statistics for the even hyperoctahedral
group and study the corresponding signed generating functions. More precisely, we show that
certain general properties that these signed generating functions have in types $A$ and $B$
(namely ``shifting'' and ``compressing'') continue to hold in type $D$. We then show that these 
generating functions factor nicely for the whole group (i.e., for the trivial quotient) and
for the maximal quotients. As a consequence of our results we
show that the signed generating function over the whole even hyperoctahedral group is
the square of the one for the symmetric group.

The organization of the paper is as follows. In the next section we recall some definitions, notation, and results that are used in the sequel. In \S 3 we define a new statistic on the even
hyperoctahedral group which is a natural analogue of the odd length statistics that have already been defined in types $A$ and $B$ in \cite{KV} and \cite{VS2}, and study 
some general properties of the corresponding signed generating functions. These 
include a complementation property, the identification of subsets of the quotients over
which the corresponding signed generating function always vanishes, and operations on a quotient that leave the corresponding signed generating function unchanged. In \S 4 we 
show that the signed generating function over the whole even hyperoctahedral group factors
nicely. As a consequence of this result we obtain that this signed generating function is the
square of the corresponding one in type $A$. In \S 5 we compute the signed generating functions of the maximal, and some other, quotients and show that these also always factor nicely. Finally, in \S 6, we present some conjectures naturally arising from
the present work, and the evidence that we have in their favor.

\section{Preliminaries}
In this section we recall some notation, definitions, and results that are used in the sequel.

\p

We let $\mathbb P:=\{1,\,2,\ldots\}$ be the set of positive integers and $\N:= \mathbb{P} \cup \{0\}$. For all $m,\,n \in \mathbb{Z}$, $m\leq n$ we let $[m,n]:= \{m,\,m+1,\ldots,\,n\}$ and $[n]:=[1,\,n]$. Given a set $I$ we denote by $|I|$ its cardinality. 
For a real number $x$ we denote by $\left\lfloor x \right\rfloor$ the greatest integer less than or equal to $x$ and by $\left\lceil x \right\rceil$ the smallest integer greater than or equal to $x$. Given $J \subseteq [0,n-1]$ there are unique integers $a_1 < \cdots < a_s$ and  $b_1 < \cdots < b_s$ such that $J = [a_1,b_1] \cup \cdots \cup [a_s,b_s]$ and  $a_{i+1} - b_{i} >1$ for $i=1, \ldots , s-1$. We call the intervals  $[a_1,b_1], \ldots , [a_s,b_s]$ the {\em connected components} of $J$.

\p

For $n_1,\ldots,\,n_k \in \N$ and $n:=\sum_{i=1}^k n_i $,  we let $\footnotesize \left[ \begin{array}{c} n \\
n_1,\ldots,n_k \end{array} \right]_{q}$  \normalsize denote the {\em $q$-multinomial coefficient}
\[\left[ \begin{array}{c} n \\
n_1,\ldots,n_k \end{array} \right]_{q}:= \frac{[n]_{q}!}{[n_1]_q !\cdot \ldots\cdot [n_k]_q !},\]
where
\[[n]_q := \frac{1-q^n}{1-q},\qquad\qquad [n]_q ! := \prod_{i=1}^{n} [i]_q \qquad\qquad\mbox{and}\qquad [0]_q!:= 1.
\]

The symmetric group $S_n$ is the group of permutations of the set $[n]$. For $\sigma \in S_n$ we use both the one-line notation $\sigma=[\sigma(1),\,\ldots,\,\sigma(n)]$ and the disjoint cycle notation.
We let $s_1,\ldots,\,s_{n-1}$ denote the standard generators of $S_n$, $s_i=(i,\,i+1)$.

The hyperoctahedral group $B_n$ is the group of signed permutations, or permutations $\sigma$ of the set $[-n,n]$ such that $\sigma(j)=-\sigma(-j)$. 
For a signed permutation $\sigma$ we use the window notation $\sigma = [\sigma(1), \ldots ,\sigma(n)]$ and the disjoint cycle notation.
The standard generating set of $B_n$ is $S=\{s_0,\,s_1,\,\ldots,\,s_{n-1}\}$, where $s_0=[-1,\,2,\,3,\ldots,\,n]$ and $s_1,\ldots,\,s_{n-1}$ are as above. 

We follow $\cite{BB}$ for notation and terminology about Coxeter groups. In particular, for $(W,S)$ a Coxeter system we let $\ell$ be the Coxeter length and for 
$I\subseteq S$ we define the quotients:
\[
W^{I} := \{w\in W \;:\;  D(w)\subseteq S\setminus I\},
\]
and 
\[
^{I}W := \{w\in W \;:\; D_L(w)\subseteq S\setminus I\},
\]
where $D(w)=\{s\in S \;:\; \ell(ws)<\ell(w)\}$, and 
$D_L(w)=\{s\in S \;:\; \ell(sw)<\ell(w)\}$. The parabolic subgroup $W_I$
is the subgroup generated by $I$.  The following result is well known (see, e.g., \cite[Proposition 2.4.4]{BB}).

\begin{pro}
Let $(W,S)$ be a Coxeter system, $J \subseteq S$, and $w \in W$. Then there exist
unique elements $w^J \in W^J$ and $w_J \in W_J$ (resp., $^Jw \in ^JW$ and 
$_Jw \in W_J$) such that $w= w^J w_J$ (resp., $_Jw ^Jw$). Furthermore
$\ell(w)= \ell(w^J)+\ell(w_J)$ (resp., $\ell(_Jw)+\ell(^Jw)$).
\end{pro}

It is well known that $S_n$ and $B_n$, with respect to the above generating 
sets, are Coxeter systems and that the following results hold (see, e.g., 
\cite[Propositions 1.5.2, 1.5.3, and \S 8.1]{BB}).
\begin{pro}
Let $\sigma \in S_n$. Then
$
\ell_A(\sigma)=| \{ (i,j) \in [n]^2 : i<j , \sigma(i) > \sigma(j) \} |
$
and
$
D(\sigma) = \{ s_i :  \sigma(i) > \sigma(i+1) \}.
$
\end{pro}

For $\sigma \in B_n$ let
\begin{align*}\p \ \inv(\sigma):=& |\{(i,j)\in [n]\times[n] \; : \; i<j,\,\sigma(i)>\sigma(j)\}|, \\ \negg(\sigma):=& |\{ i\in [n]\; : \; \sigma(i)<0\}|, \\ \nsp(\sigma):=& |\{(i,j)\in [n]\times[n] \; : \; i<j, \, \sigma(i)+\sigma(j)<0\}|.\end{align*}
\begin{pro}
Let $\sigma \in B_n$. Then
\[
\ell_B(\sigma)= \frac{1}{2} | \{ (i,j) \in [-n,n]^2 : i<j , \sigma(i) > \sigma(j) \} |=\ \inv(\sigma)+\negg(\sigma)+\nsp(\sigma)
\]
and
$
D(\sigma) = \{ s_i : i \in [0,n-1] , \sigma(i) > \sigma(i+1) \}.
$
\end{pro}

The group $D_n$ of even-signed permutations is the subgroup of $B_n$ of elements with an even number of negative entries in the window notation:
\[D_n=\{\sigma \in B_n\,:\, \negg(\sigma)\equiv 0 \pmod 2\}.\]
This is a Coxeter group of type $D_{n}$, with set of generators $S=\{s_0 ^D,s^D _1,\ldots ,s^D _{n-1}\}$, where $s_0 ^D:=[-2,-1,3,\ldots n]$ and  $ s^D_{i}:=s_i$ for $i\in [n-1]$.
 Moreover, the following holds (see, e.g., \cite[Propositions 8.2.1 and 8.2.3]{BB}).
\begin{pro}
\label{combD}
Let $\sigma \in D_n$. Then
$
\ell_D(\sigma)=\inv(\sigma)+\nsp(\sigma)
$
and
$
D(\sigma) = \{ s^D_i : i \in [0,n-1] , \sigma(i) > \sigma(i+1) \},
$
where
$\sigma(0):=\sigma(-2)$.
\end{pro}

Thus, for a subset of the generators $I\subseteq S$, that we identify with the corresponding subset $I \subseteq [0,n-1]$, we have the following description of the quotient
$$D_n^I=\{\sigma \in D_n \,:\, \sigma(i)<\sigma(i+1) \mbox{ for all } i \in I\}$$
where $\sigma(0):=-\sigma(2)$.

The following statistic was first defined in \cite{KV}. Our definition is not the 
original one, but is equivalent to it (see \cite[Definition 5.1 and Lemma 5.2]{KV})
and is the one that is best suited for our purposes.
\begin{defn}
Let $n\in {\mathbb P}$. The statistic $L_{A}:S_n \rightarrow \N$ is defined as follows. For $\sigma \in S_n$ 
\[
L_{A}(\sigma)\eqdef |\{(i,j) \in [n]^2 \; : \; i<j,\,\sigma(i)>\sigma(j),\,i\not\equiv j \pmod{2}\}|.
\]
\end{defn}

The following statistic was introduced in \cite{VS} and \cite{VS2}, and is a natural
analogue of the statistic $L_A$ introduced above, for Coxeter groups of type $B$.
\begin{defn}
\label{defLB}
Let $n\in {\mathbb P}$. The statistic $L_{B}:B_n \rightarrow \N$ is defined as follows. For $\sigma \in B_n$ 
\[
L_{B}(\sigma):= \frac{1}{2} |\{(i,j)\in [-n,\,n]^2 \, : \, i<j,\,\sigma(i)>\sigma(j),\,i\not\equiv j \pmod{2}\}|.\]
\end{defn}
\noindent For example, if $n=4$ and $\tau=[-2,4,3,-1]$ then $L_B (\tau)= \frac{1}{2}|\{(-4,-3),\,(-4,1),\,(-3,-2),$ $\,(-1,0),\,(-1,4),\,(0,1),\,(2,3),\,(3,4)\}|=4$. 

We call these statistics $L_A$ and $L_B$ the {\em odd length} of the symmetric and hyperoctahedral groups, respectively.
Note that if $\sigma \in S_n \subset B_n $ then $ L_B (\sigma)=L_A(\sigma)$. 

The odd length of an element $\sigma \in B_n$ also has a description in terms of statistics of the window notation of $\sigma$.
Given $\sigma \in B_n$ we let
\begin{align*}\p \ \oinv(\sigma):=& |\{(i,j)\in [n]^2 \; : \; i<j,\,\sigma(i)>\sigma(j),\,i\not\equiv j \pmod{2}\}|, \\ \oneg(\sigma):=& |\{ i\in [n]\; : \; \sigma(i)<0,\,i\not\equiv 0 \pmod{2}\}|, \\ \onsp(\sigma):=& |\{(i,j)\in [n]^2 \; : \; \sigma(i)+\sigma(j)<0,\,i\not\equiv j \pmod{2}\}|.\end{align*}
The following result appears in \cite[Proposition 5.1]{BC}.
\begin{pro}\label{LB}
Let $\sigma \in B_n$. Then 
$
L_B(\sigma)= \oinv(\sigma)+ \oneg(\sigma)+ \onsp(\sigma).$
\end{pro}   
The signed generating function of the odd length factors very nicely both on quotients of $S_n$ and of $B_n$. The following result was conjectured in \cite[Conjecture C]{KV} and
proved in \cite{BC}.
\begin{thm}
\label{Aquot}
Let $n \in {\mathbb P}$, $I \subseteq [n-1]$, and $I_{1}, \ldots , I_{s}$ be the connected components of $I$. Then
 \begin{align}
 \sum_{\sigma \in S_{n}^{I}} (-1)^{\ell_A (\sigma )} x^{L_A(\sigma )} &=\mc \prod_{k=2m+2}^n \left(1+(-1)^{k-1}x^{\left\lfloor\frac{k}{2}\right\rfloor}\right)
 \end{align}
where $m := \sum_{k=1}^{s} \left\lfloor \frac{|I_{k}|+1}{2} \right\rfloor $.
\end{thm}

In particular, for the whole group we have the following. 

\begin{cor}\label{wgpA}
Let $n \in {\mathbb P},\, n\geq 2$. Then
\[
\sum_{\sigma \in S_n} (-1)^{\ell_A(\sigma)}x^{L_A(\sigma)} = \prod_{i=2}^{n} \left(1+(-1)^{i-1}x^{\left\lfloor\frac{i}{2}\right\rfloor}\right).
\]
\end{cor}

For $J\subseteq [0,n-1]$ we define $J_0\subseteq J$ to be the connected component of $J$ which contains $0$, if $0\in J$, or $J_{0} := \emptyset$ otherwise. Let $J_1,\ldots,J_s$ be the remaining ordered connected components. The following result was conjectured in \cite[Conjecture 1.6]{VS2} and proved in \cite{BC} and independently in \cite{L}.

\begin{thm}
\label{Bquot}
Let $n\in \mathbb P$, $J \subseteq [0,n-1]$, and $J_0,\ldots, J_s$ be the connected components of $J$ indexed as just described.  Then
\[
\sumb=\frac{\prod\limits_{j=a+1}^{n}(1-x^j)}{\prod\limits_{i=1}^{m}(1-x^{2i})}  \mcb 
\]
where $m:=\sum_{i=1}^s \left\lfloor \frac{|J_{i}|+1}{2}\right\rfloor$ and $a:=\min\{ [0, \,n-1] \setminus J\}$.
\end{thm}

\section{Definition and general properties}

In this section we define a new statistic, on the even
hyperoctahedral group $D_n$, which is a natural analogue of the odd length statistics that have already been defined and studied in types $A$ and $B$, and study 
some of its basic properties.

Given the descriptions of $L_A$ and $L_B$ in terms of odd inversions, odd negatives and odd negative sum pairs, and the relation between the Coxeter lengths of the Weyl groups of types
$B$ and $D$ (see, e.g., \cite[Propositions 8.1.1 and 8.2.1]{BB}), the following definition is
natural.

\begin{defn}
Let $\sigma \in D_n$. We let
\[L_D(\sigma):=L_B(\sigma)- \oneg(\sigma)= \oinv(\sigma)+ \onsp(\sigma). \]
\end{defn}
For example let $n=5$, $\sigma=[2,-1,5,-4,3]$. Then $L_D(\sigma)=5$.
We call $L_D$ the {\em odd length} of type $D$. Note that the statistic $L_D$ is well defined also on $S_n$ (where it coincides with $L_A$) and on  $B_n$.
In fact, the signed distribution of $L_D$  over any quotient of $D_{n}$ and over
its ``complement'' in $B_n$, is exactly the same, as we now show. For $I \subseteq [0,n-1]$
let $(B_{n} \setminus D_{n})^{I}\eqdef \{ \sigma \in B_n \setminus D_n : \; \sigma (i) <
\sigma (i+1) \mbox{ for all }  i \in I \}$ where  $\sigma (0) \eqdef - \sigma (2)$. Note that 
$(B_n \setminus D_n)^{I}=B_{n}^{I} \setminus D_n^I$ if $I \subseteq [n-1]$.

\begin{lem}\label{compl}
Let $n\in \mathbb P$ and $I \subseteq [0,n-1]$. Then
\[\sum_{\sigma \in D_n ^{I}}{y^{\ell_D(\sigma)}x^{L_D(\sigma)}}=\sum_{\sigma \in (B_n \setminus D_n)^{I}}{y^{\ell_D(\sigma)}x^{L_D(\sigma)}}. \]
In particular, $\sum_{\sigma \in D_n ^{I}}{(-1)^{\ell_D(\sigma)}x^{L_D(\sigma)}}=\sum_{\sigma \in (B_n \setminus D_n)^{I}}{(-1)^{\ell_D(\sigma)}x^{L_D(\sigma)}}$. \end{lem}
\begin{proof}
Left multiplication by $s_0$ (that is, changing the sign of $1$ in the window notation) is a bijection between $D_n^{I}$ and $(B_n\setminus D_n)^{I}$. Moreover, (odd) inversions and  (odd) negative sum pairs are preserved by this operation so
$L_D(s_0 \sigma )=L_D(\sigma)$, and $ \ell_D(s_0 \sigma )=\ell_D(\sigma)$,
for all $ \sigma \in D_{n}$ and the result follows.
\end{proof}
In what follows, since we are mainly concerned with distributions in type $D$, we omit the subscript and write just $\ell$ and $L$ for the length and odd length, respectively, on $D_n$.
We now  show that the generating function of $(-1)^{\ell(\cdot)}x^{L(\cdot)}$ 
over any quotient of $D_{n}$ such that $s_{0}^{D} \in D_{n}^{I}$ can be reduced to elements for which the maximum (or the minimum) is in certain positions. More precisely, we prove that, for a given quotient, our generating function is zero over all elements for which the maximum (or minimum) is sufficiently far from $I$. For a subset $I\subseteq [0,n-1]$ we let $\delta_0 (I)=1$ if $0 \in I$ and $\delta_0 (I)=0$ otherwise.
\begin{lem}\label{zerod}
Let $n\in \mathbb P$, $n\geq 3$, $I\subseteq [0,n-1]$. Let $a \in  [2+\delta_0 (I),n-1]$ be such that $[a-2,a+1]\cap I=\emptyset$. Then
\[\sum_{\substack{\{\sigma \in D_n^I :\\\sigma(a)=n\}}}\f =\sum_{\substack{\{\sigma \in D_n^I :\\\sigma(a)=-n\}}}\f=0.\]
\end{lem}
\begin{proof}
In our hypotheses, if $\sigma \in D_n ^I$ then also $\sigma^{a}:= \sigma (-a-1,-a+1)
 (a-1,a+1)$ is in the same quotient. Clearly $(\sigma^a)^a=\sigma$ and $|\ell(\sigma)-\ell(\sigma^a)|=1$, while, since $\sigma(a)=n$, $L(\sigma^a)=L(\sigma)$. Therefore we have that
\begin{align*}\sum_{\substack{\{\sigma \in  {D}^{I}_{n}: \\  \sigma(a)=n \} }} (-1)^{\ell (\sigma )}x^{L(\sigma )}&=& \sum_{\substack{\{\sigma \in  {D}^{I}_n: \sigma(a)=n, \\ \sigma (a-1) < \sigma (a+1)  \} }}\left( (-1)^{\ell (\sigma )}x^{L(\sigma )} + (-1)^{\ell (\sigma^a )}x^{L(\sigma^a )}\right) =0.\end{align*} 
The proof of the second equality is exactly analogous and is therefore omitted. 
\end{proof}

Although we do not know of any definition of our (or of any other) odd length statistics
in Coxeter theoretic language, it is natural to expect that the only automorphism of
the Dynkin diagram of $D_n$ preserves the corresponding signed generating function.
This is indeed the case, as we now show.

\begin{pro}\label{zerouno}
Let $n \in \mathbb P$, $n\geq 2$, and $I \subseteq [2,n-1]$. Then
\[
\sum_{\sigma \in D_n ^{I \cup \{0\}}}{y^{\ell(\sigma)}x^{L(\sigma)}}=
\sum_{\sigma \in D_n ^{I \cup \{1\}}}{y^{\ell(\sigma)}x^{L(\sigma)}}.
\]
In particular, $\sum_{\sigma \in D_n ^{I \cup \{0\}}}{(-1)^{\ell(\sigma)}x^{L(\sigma)}}=\sum_{\sigma \in D_n ^{I \cup \{1\}}}{(-1)^{\ell(\sigma)}x^{L(\sigma)}}$.
\end{pro}
\begin{proof}
Right multiplication by $s_{0}$ (i.e., changing the sign of the leftmost element in the window notation) is a bijection between $D_{n}
^{I \cup \{ 0 \}}$ and $(B_{n} \setminus D_{n} )^{I \cup \{ 1 \}}$. Furthermore, if $\sigma \in D_{n}$, then
\begin{eqnarray*} 
\oinv (\sigma s_{0}) & = & \oinv (\sigma ) - |\{ i \in [2,n]: i \equiv 0 \pmod{2}, \; \sigma (1)>\sigma (i) \} | \\
& & + |\{ i \in [2,n]: i \equiv 0 \pmod{2}, \; -\sigma (1) > \sigma (i) \} | ,\\
 \onsp (\sigma s_{0}) & = &  \onsp (\sigma ) - |\{ i \in [2,n]: i \equiv 0 \pmod{2}, \; \sigma (1) + \sigma (i) <0 \} | \\
& & + | \{ i \in [2,n]: i \equiv 0 \pmod{2}, \; -\sigma (1)+ \sigma (i) <0 \} | ,\\ 
\inv (\sigma s_{0} )& =& \inv (\sigma ) -|\{ i \in [2,n]: \; \sigma (1) > \sigma (i) \} | + |\{ i \in [2,n]: \; - \sigma (1)>
\sigma (i) \} |,\\
\mbox{and} \quad \qquad& &
\\
\nsp  (\sigma s_{0} ) &=& \nsp  (\sigma ) -|\{ i \in [2,n]: \; \sigma (1) + \sigma (i)<0 \} | + |\{ i \in [2,n]: \; - \sigma (1)+
\sigma (i) <0 \} |. \end{eqnarray*} 
Therefore $L(\sigma s_{0})=L(\sigma )$ and $\ell (\sigma s_{0})=\ell (\sigma )$. Hence
\[
\sum_{\sigma \in (B_{n} \setminus D_{n}) ^{I \cup \{ 1\}}}y^{\ell (\sigma )} \; x^{L(\sigma )} =
\sum _{\sigma \in D_{n}^{I \cup \{ 0 \}}}y^{\ell (\sigma s_{0})} \; x^{L(\sigma s_{0})} =
\sum _{\sigma \in D_{n}^{I \cup \{ 0 \}}}y^{\ell (\sigma )} \; x^{L(\sigma )} ,
\]
and the result follows from Lemma \ref{compl}. 
\end{proof}

We conclude by showing
that when $I$ does not contain $0$, each connected component can be shifted to the left or to the right, as long as it remains a connected component, without changing the generating function over the corresponding quotient. The proof is identical to that of \cite[Proposition 3.3]{BC}, and is therefore omitted.
\begin{pro}\label{shd}
Let $I\subseteq [n-1]$, $i\in \mathbb P$, $k\in \N$ be such that $[i,\,i+2k]$ is a connected component of $I$ and $i+2k+2 \notin I$. Then
\[
\sum_{\sigma \in D_n ^I}{(-1)^{\ell(\sigma)}x^{L(\sigma)}}=\sum_{\sigma \in D_n ^{I\cup \tilde{I}}}{(-1)^{\ell(\sigma)}x^{L(\sigma)}}=\sum_{\sigma \in D_n ^{\tilde{I}}}{(-1)^{\ell(\sigma)}x^{L(\sigma)}}
\]
where $\tilde{I}:=(I\setminus\{i\})\cup\{i+2k+1\}$.
\end{pro}
Shifting is also allowed when $I$ contains $0$, but only for connected components which are sufficiently far from it, as stated in the next result.
\begin{pro}
Let $I\subseteq [0,n-1]$, $i\in \mathbb P$, $i>2$ and $k\in \N$ such that $[i,\,i+2k]$ is a connected component of $I$ and $i+2k+2 \notin I$. Then
\[
\sum_{\sigma \in D_n ^I}{(-1)^{\ell(\sigma)}x^{L(\sigma)}}=\sum_{\sigma \in D_n ^{I\cup \tilde{I}}}{(-1)^{\ell(\sigma)}x^{L(\sigma)}}=\sum_{\sigma \in D_n ^{\tilde{I}}}{(-1)^{\ell(\sigma)}x^{L(\sigma)}}
\]
where $\tilde{I}:=(I\setminus\{i\})\cup\{i+2k+1\}$.
\end{pro}
\begin{proof}
The proof is analogous to that of \cite[Proposition 3.3]{BC} noting that , since $i>2$, $\sigma \in D_n^{I}$ if and only if $\sigma(i+2j,i+2k+2)(-i-2j,-i-2k-2)\in D_n^I$.
\end{proof}

\section{Trivial quotient}

In this section, using the results in the previous one, we compute the generating function of $(-1)^{\ell(\cdot)}x^{L(\cdot)}$ over the whole even hyperoctahedral group $D_n$. In particular, we obtain that this generating function is
the square of the corresponding one in type $A$ (i.e., for the symmetric group).

\begin{thm}\label{sq}
Let $n\in \mathbb{P}$, $n\geq 2$. Then
\[
\sum_{\sigma \in D_n}{(-1)^{\ell (\sigma)}x^{L(\sigma)}}=\prod_{j=2}^{n}(1+(-1)^{j-1}x^{\left\lfloor\frac{j}{2}\right\rfloor})^2 .
\]
\end{thm}
\begin{proof}
We proceed by induction on $n$. By Lemma \ref{zerod}, the sum over all elements for which $n$ or $-n$ appears in positions different from $1$ and $n$ is zero. So the generating function over $D_n$ reduces to
%\small
\begin{align*}
\sum_{\sigma \in D_n} {(-1)^{\ell(\sigma)}x^{L(\sigma)}}&=
\sum_{\substack{ \{\sigma \in  D_n: \\ \sigma(1)=n\}}} 
{(-1)^{\ell(\sigma)}x^{L(\sigma)}} +\sum_{\substack{ \{\sigma \in  D_n: \\ \sigma(n)=n\} }
}{(-1)^{\ell(\sigma)}x^{L(\sigma)}}
\\
&+
\sum_{\substack{ \{\sigma \in  D_n: \\ \sigma(1)=-n\} }}{(-1)^{\ell(\sigma)}x^{L(\sigma)}}+\sum_{\substack{\{\sigma \in  D_n: \\ \sigma(n)=-n\} 
}}{(-1)^{\ell(\sigma)}x^{L(\sigma)}} \\
&= \sum_{\sigma \in D_{n-1}}{(-1)^{\ell(\tilde{\sigma})}x^{L(\tilde{\sigma})}} +
\sum_{\sigma \in D_{n-1}}{(-1)^{\ell(\sigma)}x^{L(\sigma)}} \\
&+ \!\!\!\!\sum_{\sigma \in B_{n-1}\setminus D_{n-1}}\!\!{(-1)^{\ell(\hat{\sigma})}x^{L(\hat{\sigma})}}+\!\!\!\!\sum_{\sigma \in B_{n-1}\setminus D_{n-1}}\!\!{(-1)^{\ell(\check{\sigma})}x^{L(\check{\sigma})}} ,
\end{align*}
\normalsize
where $\tilde\sigma:=[n,\sigma(1),\ldots,\sigma(n-1)]$, $\hat\sigma:=[-n,\sigma(1),\ldots,\sigma(n-1)],$
 and $\check\sigma:=[\sigma(1),\ldots,\sigma(n-1),-n]$.
But, by our definition and Proposition 8.2.1 of \cite{BB}, we have that
\begin{align}
L(\tilde\sigma)&%=L(\sigma)+\left\lceil\frac{n-1}{2}\right\rceil
=L(\sigma)+m, \qquad \ell(\tilde\sigma)=\ell(\sigma)+n-1\\ \label{ha}
L(\hat\sigma)&%=L(\sigma)+\left\lceil\frac{n-1}{2}\right\rceil
=L(\sigma)+m, \qquad \ell(\hat\sigma)=\ell(\sigma)+n-1\\ \label{che}
L(\check\sigma)&%=L(\sigma)+2\left\lceil\frac{n-1}{2}\right\rceil
=L(\sigma)+2m,\:\: \quad \ell(\check\sigma)=\ell(\sigma)+2(n-1),
\end{align}
where $m \eqdef \left\lfloor \frac{n}{2} \right\rfloor $. Therefore
\[\sum_{\sigma \in D_{n-1}}{(-1)^{\ell(\tilde{\sigma})}x^{L(\tilde{\sigma})}}=(-1)^{n-1}x^m\sum_{\sigma \in D_{n-1}}{(-1)^{\ell({\sigma})}x^{L({\sigma})}}\]
and, similarly,
\begin{eqnarray*}
\sum_{\sigma \in B_{n-1}\setminus D_{n-1}}{(-1)^{\ell(\hat{\sigma})}x^{L(\hat{\sigma})}}&=&(-1)^{n-1}x^m\sum_{\sigma \in B_{n-1}\setminus D_{n-1}}{(-1)^{\ell({\sigma})}x^{L({\sigma})}}\\
\sum_{\sigma \in B_{n-1}\setminus D_{n-1}}{(-1)^{\ell(\check{\sigma})}x^{L(\check{\sigma})}}&=&x^{2m}\sum_{\sigma \in B_{n-1}\setminus D_{n-1}}{(-1)^{\ell({\sigma})}x^{L({\sigma})}}.
\end{eqnarray*}
\vspace{3mm}

So by  Lemma \ref{compl} and our induction hypothesis we obtain that
\begin{eqnarray*}
\sum_{\sigma \in D_n} {(-1)^{\ell(\sigma)}x^{L(\sigma)}}&=&(1+(-1)^{n-1}x^m) \sum_{\sigma \in D_{n-1}}{(-1)^{\ell(\sigma)}x^{L(\sigma)}} +\\
&+&((-1)^{n-1}x^m +x^{2m})\!\!\!\!\!\sum_{\sigma \in B_{n-1}\setminus D_{n-1}}{\!\!\!\!\!\!(-1)^{\ell(\sigma)}x^{L(\sigma)}}\\
&=&(1+2(-1)^{n-1}x^{m} +x^{2m})\sum_{\sigma\in D_{n-1}}{(-1)^{\ell(\sigma)}x^{L(\sigma)}}\\
&=&\left(1+(-1)^{n-1}x^{\left\lfloor \frac{n}{2} \right\rfloor }\right)^2\sum_{\sigma \in D_{n-1}} (-1)^{\ell(\sigma)}x^{L(\sigma)},\\
\end{eqnarray*}
and the result follows by induction.
\end{proof}

As an immediate consequence of Theorem \ref{sq} and of Corollary \ref{wgpA} we obtain the
following result.
\begin{cor}
\label{DA}
Let $n \in \mathbb{P}$, $n \geq 2$. Then
\[ \sum _{\sigma \in D_{n}} (-1)^{\ell (\sigma )} \, x^{L}(\sigma ) =
\left( \sum _{\sigma \in S_{n}} (-1)^{\ell _{A}(\sigma )} \, x^{L_{A}(\sigma )} \right)^{2}.
\Box \]
\end{cor}
It would be interesting to have a direct proof of Corollary \ref{DA}.

\vspace{3mm}
\section{Maximal and other quotients}
In this section we compute, using the results in \S 3, the signed generating function of the odd length over the maximal, and some other,
quotients of $D_{n}$. In particular, we obtain that these generating functions always factor nicely.

\begin{thm}\label{maxquod}
Let $n\in \mathbb P$, $n\geq 3$ and $i\in [0,n-1]$. Then
\[
\sum_{\sigma \in D_n^{\{i\}}}{ (-1)^{\ell(\sigma)}x^{L(\sigma)}}=(1-x^2)\prod_{j=4}^n (1+(-1)^{j-1} x^{\left\lfloor\frac{j}{2} \right\rfloor})^2.%=(1-x^2) \,\left(\sum_{\sigma \in S_n^{ \{ i \}}}{(-1)^{\ell_A(\sigma)}x^{L_A(\sigma)}}\right)^2.
\]
\end{thm}
\begin{proof}
By Propositions \ref{zerouno} (with $I = \emptyset$) and \ref{shd}, we may assume $i=1$. We proceed by induction on $n \geq 3$.
  By Lemma \ref{zerod} we have that the sum over $\sigma\in D^{\{1\}}_n$ such that $n$ or $-n$ appear in the window in any position but $1,2,3$, or $n$ is zero. Furthermore, if $\sigma \in D_n^{\{1\}}$ then $\sigma^{-1}(n)\neq 1$ and $\sigma^{-1}(-n)\neq 2$. Moreover, the map $\sigma
\mapsto \sigma (1,3) (-1,-3)$ is a bijection of $\{ \sigma \in D_{n}^{\{ 1\} }: \; \sigma (2)
=n \}$ in itself. But $L (\sigma (1,3)(-1,-3))=L(\sigma )$ and $\ell 
(\sigma ) \not \equiv \ell (\sigma (1,3)(-1,-3)) \pmod{2} $ for all $\sigma \in D_{n}^{\{
1\} }$ such that $\sigma (2)=n$  so the sum is zero also over this kind of elements. Thus we have that:
\begin{align*}
&\sum_{\sigma \in D_n^{\{1\}}}{ (-1)^{\ell(\sigma)}x^{L(\sigma)}}=\sum_{\substack{\{\sigma \in  D_n^{\{1\}}: \\ \sigma(3)=n\}}}{ (-1)^{\ell(\sigma)}x^{L(\sigma)}}+\sum_{\substack{ \{\sigma \in  D_n^{\{1\}}: \\ \sigma(n)=n\}}}{ (-1)^{\ell(\sigma)}x^{L(\sigma)}} +\\
&+ \sum_{\substack{ \{\sigma \in  D_n^{\{1\}}: \\ \sigma(1)=-n\}}}{ (-1)^{\ell(\sigma)}x^{L(\sigma)}}+\sum_{\substack{ \{\sigma \in  D_n^{\{1\}}: \\ \sigma(3)=-n\}}}{ (-1)^{\ell(\sigma)}x^{L(\sigma)}}+\sum_{\substack{\{\sigma \in  D_n^{\{1\}}: \\ \sigma(n)=-n\}}}{ (-1)^{\ell(\sigma)}x^{L(\sigma)}}
\end{align*}
\begin{align*}&=\sum_{\substack{ \{\sigma \in  D_n^{\{1\}}: \\ \sigma(3)=n\}}}{ (-1)^{\ell(\sigma)}x^{L(\sigma)}}+\sum_{\sigma \in  D_{n-1}^{\{1\}}}{ (-1)^{\ell(\sigma)}x^{L(\sigma)}} + \sum_{\sigma \in  B_{n-1}\setminus D_{n-1}}{ (-1)^{\ell(\hat\sigma)}x^{L(\hat\sigma)}}+\\
&+\sum_{\substack{\{\sigma \in  D_n^{\{1\}}: \\ \sigma(3)=-n\}}}{ (-1)^{\ell(\sigma)}x^{L(\sigma)}}+\sum_{\sigma \in (B_{n-1}\setminus D_{n-1})^{\{ 1\}}}{ (-1)^{\ell(\check\sigma)}x^{L(\check\sigma)}}
\end{align*}
where $\hat\sigma:=[-n,\sigma(1),\ldots,\sigma(n-1)]$ and $\check\sigma:=[\sigma(1),\ldots,\sigma(n-1),-n]$. Now
\[\sum_{\substack{\{\sigma \in  D_n^{\{1\}}: \\ \sigma(3)=n\}}}{ (-1)^{\ell(\sigma)}x^{L(\sigma)}}=\sum_{\substack{ \{\sigma \in  D_{n-1}: \\ \sigma(1)>\sigma(2)\}}}{ (-1)^{\ell(\bar\sigma)}x^{L(\bar\sigma)}}\]
where $\bar\sigma:=[\sigma(2),\sigma(1),n,\sigma(3),\ldots,\sigma(n-1)]$. But
$\ell(\bar\sigma)= \inv(\sigma)+n-4+\nsp(\sigma)=\ell(\sigma)+n-4$, and
$L(\bar\sigma)= \oinv(\sigma)-1+\left\lceil\frac{n-3}{2}\right\rceil+ \onsp(\sigma)=L(\sigma)+\left\lceil\frac{n-5}{2}\right\rceil=L(\sigma)+m-2,$
	where $m \eqdef \left\lceil \frac{n-1}{2} \right\rceil$, so
\begin{align*} &\sum_{\substack{\{\sigma \in  D_n^{\{1\}}: \\ \sigma(3)=n\}}}{ (-1)^{\ell(\sigma)}x^{L(\sigma)}}=(-1)^n x^{m-1}\sum_{\substack{ \{\sigma \in  D_{n-1}: \\ \sigma(1)>\sigma(2)\}}}{ (-1)^{\ell(\sigma)}x^{L(\sigma)}}\\
&=(-1)^n x^{m-2}\left(\sum_{\sigma \in D_{n-1}}\f -\sum_{\sigma \in D_{n-1}^{\{1\}}} \f\right) .
\end{align*}
Similarly,
\[\sum_{\substack{ \{\sigma \in  D_n^{\{1\}}: \\ \sigma(3)=-n\}}}{ (-1)^{\ell(\sigma)}x^{L(\sigma)}}=\sum_{\substack{ \{\sigma \in B_{n-1}\setminus D_{n-1}: \\ \sigma(1)>\sigma(2)\}}}{ (-1)^{\ell(\bar{\bar{\sigma}})}x^{L(\bar{\bar{\sigma}})}}\]
where $\bar{\bar{\sigma}}:=[\sigma(2),\sigma(1),-n,\sigma(3),\ldots,\sigma(n-1)]$ and
$\ell(\bar{\bar{\sigma}})= \inv(\sigma)+1+\nsp(\sigma)+n-1=\ell(\sigma)+n$,  $L(\bar{\bar{\sigma}})= \oinv(\sigma)+ \onsp(\sigma)+1+\left\lceil\frac{n-3}{2}\right\rceil=L(\sigma)+\left\lceil\frac{n-1}{2}\right\rceil=L(\sigma)+m$.
So, by Lemma \ref{compl},
\begin{align*}
&\sum_{\substack{ \{\sigma \in  D_n^{\{1\}}: \\ \sigma(3)=-n\}}}{ (-1)^{\ell(\sigma)}x^{L(\sigma)}}=(-1)^n x^{m} \sum_{\substack{ \{\sigma \in B_{n-1}\setminus D_{n-1}: \\ \sigma(1)>\sigma(2)\}}}{ (-1)^{\ell(\sigma)}x^{L(\sigma)}}\\
&=(-1)^{n}x^{m}\left(\sum_{\sigma\in B_{n-1} \setminus D_{n-1}}{(-1)^{\ell(\sigma)}x^{L(\sigma)}}-\!\!\!\!\!\sum_{ (B_{n-1} \setminus D_{n-1})^{\{ 1\}}}{\!\!\! (-1)^{\ell(\sigma)}x^{L(\sigma)}}\right)\\
&=(-1)^{n}x^{m}\left(\sum_{\sigma \in D_{n-1}}{(-1)^{\ell(\sigma)}x^{L(\sigma)}}-\sum_{ \sigma \in D_{n-1}^{\{1\}}}{ (-1)^{\ell(\sigma)}x^{L(\sigma)}}\right).
\end{align*}
Moreover, by  \eqref{ha} and \eqref{che} we have
\[\sum_{\sigma \in B_{n-1}\setminus D_{n-1}} {(-1)^{\ell(\hat\sigma)}x^{L(\hat\sigma)}}=(-1)^{n-1}x^m \sum_{\sigma \in B_{n-1}\setminus D_{n-1}} {(-1)^{\ell(\sigma)}x^{L(\sigma)}}\]
and
\[
\sum_{\sigma \in (B_{n-1} \setminus D_{n-1})^{\{ 1\}}} {(-1)^{\ell(\check\sigma)}x^{L(\check\sigma)}}=x^{2m} \sum_{\sigma \in (B_{n-1}\setminus D_{n-1})^{\{ 1 \}}} {(-1)^{\ell(\sigma)}x^{L(\sigma)}}.
\]

Thus we get, again  by Lemma \ref{compl}, 
\small
\begin{eqnarray*}
\sum_{\sigma \in D_n^{\{1\}}}{ (-1)^{\ell(\sigma)}x^{L(\sigma)}}&=&
(-1)^n x^{m-2}\left(\sum_{\sigma \in  D_{n-1}}{ (-1)^{\ell(\sigma)}x^{L(\sigma)}}-\sum_{\sigma \in  D_{n-1}^{\{1\}}}{ (-1)^{\ell(\sigma)}x^{L(\sigma)}}\right)\\
&+&\sum_{\sigma \in  D_{n-1}^{\{1\}}}{ (-1)^{\ell(\sigma)}x^{L(\sigma)}}+(-1)^{n-1}x^{m}\sum_{\sigma \in  D_{n-1}}{ (-1)^{\ell(\sigma)}x^{L(\sigma)}}\\
%&+&(-1)^{n-1}x^{m}\sum_{\sigma \in  D_{n-1}}{ (-1)^{\ell(\sigma)}x^{L(\sigma)}}++\\
&+&(-1)^{n}x^{m}\left(\sum_{\sigma \in D_{n-1}}{(-1)^{\ell(\sigma)}x^{L(\sigma)}}-\sum_{\sigma \in  D_{n-1}^{\{1\}}}{ (-1)^{\ell(\sigma)}x^{L(\sigma)}}\right) \\
&+&x^{2m}\sum_{\sigma \in D_{n-1}^{\{ 1 \}}}{ \!\!\!(-1)^{\ell(\sigma)}x^{L(\sigma)}}\\
&=&(-1)^n x^{m-2}\sum_{\sigma \in  D_{n-1}}{ (-1)^{\ell(\sigma)}x^{L(\sigma)}}\\
&+&\left(1+(-1)^{n-1}x^{m-2}+(-1)^{n-1}x^{m}+x^{2m}\right)\sum_{\sigma\in D^{\{1\}}_{n-1}}{(-1)^{\ell(\sigma)}x^{L(\sigma)}}
\end{eqnarray*}

\normalsize
and the result follows by Theorem \ref{sq} and our induction hypothesis.
\end{proof}

We note the following consequence of Theorems \ref{sq} and \ref{maxquod}.
\begin{cor}
Let $n \in \mathbb{P}$, $n \geq 3$, and $i \in [0,n-1]$. Then
\[ \sum _{\sigma \in D_{n}}(-1)^{\ell (\sigma )} \, x^{L(\sigma )}=(1-x^{2}) \, \sum _{\sigma
\in D_{n}^{\{ i\}}}(-1)^{\ell (\sigma )} \, x^{L(\sigma )} . \]
\end{cor}
\begin{proof}
This follows immediately from Theorems \ref{sq} and \ref{maxquod}. 
\end{proof}

The results obtained up to now compute
$\sum_{\sigma \in D_{n}^{I}} (-1)^{\ell (\sigma )} x^{L(\sigma )}$ when $|I| \leq 1$.
A natural next step is to try to compute these generating functions if $|I \setminus \{ 0 \}| \leq 1$. We are able to do this for $I=\{ 0,1 \}$, and $I=\{ 0,2 \}$. 
The computation
for $I=\{ 0,2 \}$ follows easily from results that we have already obtained.

\begin{cor}
\label{pro02}
Let $n \in {\mathbb P}$, $n \geq 4$. Then 
\[\sum_{\sigma \in
D_{n}^{\{ 0, \, 2 \}}}{(-1)^{\ell(\sigma)}x^{L(\sigma)}} = 
(1-x^2) \prod_{j=4}^n (1+(-1)^{j-1} x^{\left\lfloor\frac{j}{2} \right\rfloor})^2.
%=(1-x^2) \, (S_n^{ \{ 2 \} } (x))^2.
\]
\end{cor}
\begin{proof}
By Proposition \ref{zerouno}  and Proposition \ref{shd} we have
\[\sum_{\sigma \in
D_{n}^{\{ 0,2 \}}}{(-1)^{\ell(\sigma)}x^{L(\sigma)}} =\sum_{\sigma \in
D_{n}^{\{ 1,2 \}}}{(-1)^{\ell(\sigma)}x^{L(\sigma)}} = \sum_{\sigma \in
D_{n}^{\{ 1 \}}}{(-1)^{\ell(\sigma)}x^{L(\sigma)}},   \]
and the result follows by Theorem \ref{maxquod}.
\end{proof}

We conclude this section by computing
$\sum_{\sigma \in D_{n}^{I}} (-1)^{\ell (\sigma )} x^{L(\sigma )}$ when $I =\{ 0,1 \}$.

\begin{thm}
\label{pro01}
Let $n \in {\mathbb P}$, $n \geq 3$. Then 
\[\sum_{\sigma \in
D_{n}^{\{ 0, \, 1 \}}}{(-1)^{\ell(\sigma)}x^{L(\sigma)}} =
(1+x^2) \prod_{j=4}^n (1+(-1)^{j-1} x^{\left\lfloor\frac{j}{2} \right\rfloor})^2.%=
%(1+x^2) \, (S_n^{ \{ 1 \} } (x))^2.
\]
\end{thm}
\begin{proof}
 We proceed by induction on $n \geq 3$.
  By Lemma \ref{zerod} we have that the sum over $\sigma\in D^{\{0,\,1\}}_n$ such that $n$ or $-n$ appear in the window in any position but $1,2,3$, or $n$ is zero; moreover for $\sigma \in D_n^{\{0,\,1\}}$ we always have $\sigma^{-1}(\pm n)\neq 1$ and $\sigma^{-1}(-n)\neq 2$. Also, the map $\sigma
\mapsto \sigma (1,3) (-1,-3)$ is a bijection of $\{ \sigma \in D_{n}^{\{ 0,\,1\} }: \, \sigma (2)
=n \}$ in itself, so by the same argument as in the proof of Theorem \ref{maxquod}  the sum over this kind of elements  is zero. Thus we have that:
\begin{align*}
\sum_{\sigma \in D_n^{\{0,\,1\}}}{ (-1)^{\ell(\sigma)}x^{L(\sigma)}}&=
\sum_{\substack{\{\sigma \in  D_n^{\{0,\,1\}}: \\ | \sigma(3) |=n\}}}{ (-1)^{\ell(\sigma)}x^{L(\sigma)}}+
\sum_{\substack{ \{\sigma \in  D_n^{\{0,\,1\}}: \\ | \sigma(n) |=n\}}}{ (-1)^{\ell(\sigma)}x^{L(\sigma)}}.
\end{align*}
By \eqref{che} and Lemma \ref{compl} we have that
\begin{align*}
\sum_{\substack{ \{\sigma \in  D_n^{\{0,\,1\}}: \\ | \sigma(n) |= n\}}}{ (-1)^{\ell(\sigma)}x^{L(\sigma)}} & = \left(1+x^{2\left\lfloor\frac{n}{2}\right\rfloor}\right)\sum_{\sigma \in  D_{n-1}^{\{0,\,1\}}}{ (-1)^{\ell(\sigma)}x^{L(\sigma)}} \\
& = \left(1+x^{2\left\lfloor\frac{n}{2}\right\rfloor}\right) (1+x^2) \prod_{j=4}^{n-1} (1+(-1)^{j-1} x^{\left\lfloor\frac{j}{2} \right\rfloor})^2
\end{align*}
by our induction hypothesis. Moreover,
\[
\sum_{\substack{\{\sigma \in  D_n^{\{0,\,1\}}: \\ \sigma(3)=n\}}}{ (-1)^{\ell(\sigma)}x^{L(\sigma)}}=\sum_{\sigma \in D_{n-1}^{\{1\}}\setminus D_{n-1}^{\{0,\,1\}}}{ (-1)^{\ell(\bar\sigma)}x^{L(\bar\sigma)}}
\]
where $\bar\sigma:=[-\sigma(2),-\sigma(1),n,\sigma(3),\ldots,\sigma(n-1)]$. 
But
\begin{eqnarray*} 
\inv (\bar\sigma )& =& \inv ([n,\sigma (3), \ldots, \sigma (n-1)]) + |\{ j \in [3,n-1]: \; - \sigma (1) > \sigma (j) \} | \\ & & + |\{ j \in [3,n-1]: \; - \sigma (2)>
\sigma (j) \} |,\\
\nsp  (\bar\sigma ) &=& \nsp  ([n,\sigma (3), \ldots, \sigma (n-1)]) + 
|\{ j \in [3,n-1]: \; - \sigma (1) + \sigma (j) <0 \} | \\ & & + 
|\{ j \in [3,n-1]: \; - \sigma (2)+ \sigma (j) <0 \} |, \\
\oinv (\bar\sigma ) & = & \oinv ([n,\sigma (3), \ldots, \sigma (n-1)] ) + |\{ j \in [3,n-1]: j \equiv 0 \pmod{2}, \; \sigma (1)+\sigma (j) <0 \} | \\
& & + |\{ j \in [3,n-1]: j \equiv 1 \pmod{2}, \; \sigma (2) + \sigma (j) <0 \} | ,\\ 
\mbox{and} \quad \qquad& &
\\
 \onsp (\bar\sigma ) & = &  \onsp ([n,\sigma (3), \ldots, \sigma (n-1)] ) + |\{ j \in [3,n-1]: j \equiv 0 \pmod{2}, \; \sigma (1) > \sigma (j) \} | \\
& & + | \{ j \in [3,n-1]: j \equiv 1 \pmod{2}, \; \sigma (2) > \sigma (j) \} |. \end{eqnarray*} 
Therefore, 
$\ell(\bar\sigma)=\ell(\sigma)+n-4$, and
$L(\bar\sigma)= \oinv(\sigma)-1+\left\lceil\frac{n-3}{2}\right\rceil+ \onsp(\sigma)=L(\sigma)+m-2,$
	where $m \eqdef \left\lfloor \frac{n}{2} \right\rfloor$.
Similarly,
\[
\sum_{\substack{ \{\sigma \in  D_n^{\{0,\,1\}}: \\ \sigma(3)=-n\}}}{ (-1)^{\ell(\sigma)}x^{L(\sigma)}}=\sum_{\substack{ \{\sigma \in B_{n-1}\setminus D_{n-1}: \\ \sigma(1)<\sigma(2)<-\sigma(1)\}}}{ (-1)^{\ell(\bar{\bar{\sigma}})}x^{L(\bar{\bar{\sigma}})}}
\]
where $\bar{\bar{\sigma}}:=[-\sigma(2),-\sigma(1),-n,\sigma(3),\ldots,\sigma(n-1)]$ and
$\ell(\bar{\bar{\sigma}})=\ell(\sigma)+n$,  $L(\bar{\bar{\sigma}})= \oinv(\sigma)+ \onsp(\sigma)+1+\left\lceil\frac{n-3}{2}\right\rceil=L(\sigma)+m$.
But $\{ \sigma \in B_{n-1} \setminus D_{n-1} \, : \,\sigma (1) < \sigma (2) < - \sigma (1) \} = (B_{n-1} \setminus D_{n-1})^{ \{ 1 \} } \setminus (B_{n-1} \setminus D_{n-1})^{ \{ 0,1 \} }$, so by Lemma \ref{compl}, Theorem \ref{maxquod} and our induction hypothesis
\begin{align*}
\hspace{-2.5em}\sum_{\substack{ \{\sigma \in  D_n^{\{0,\,1\}}: \\ | \sigma(3) |= n\}}}\!\!\!{ (-1)^{\ell(\sigma)}x^{L(\sigma)}}& = (-1)^n x^ {m-2} (1+x^2) \left(\sum_{\sigma \in  D_{n-1}^{\{1\}}}{\! (-1)^{\ell(\sigma)}x^{L(\sigma)}}-\!\!\sum_{\sigma \in  D_{n-1}^{\{0,\,1\}}}\!\!{ (-1)^{\ell(\sigma)}x^{L(\sigma)}}\right)\\
&=2(-1)^{n-1}x^m (1+x^2)  \prod_{j=4}^{n-1} (1+(-1)^{j-1} x^{\left\lfloor\frac{j}{2} \right\rfloor})^2
\end{align*}

Thus
\[
\sum_{\sigma \in D_{n}^{\{ 0, \, 1 \}}}{(-1)^{\ell(\sigma)}x^{L(\sigma)}} =
(1+(-1)^{n-1}x^m)^2 (1+x^2) \prod_{j=4}^{n-1} (1+(-1)^{j-1} x^{\left\lfloor\frac{j}{2} \right\rfloor})^2 
\]

and the result follows.
\end{proof}

\section{Open problems}

In this section we present some conjectures naturally arising from the present work
and the evidence that we have in their favor.

In this paper we have given closed product formulas for $\sum_{\sigma \in D_{n}^{I}} (-1)^{\ell (\sigma )} x^{L(\sigma )}$ when $|I| \leq 1$, $I=\{0,1\}$ and $I=\{0,2\}$.

We feel that such formulas always exist. In particular, if $|I \setminus \{ 0,1 \}| \leq 1$, we feel that the following holds. For $n \in {\mathbb P}$ and $I \subseteq [0,n-1]$ 
let, for brevity, $D_{n}^{I}(x) \eqdef \sum_{\sigma \in D_{n}^{I}} (-1)^{\ell (\sigma )} x^{L (\sigma )}$.

\begin{con}
\label{con0i}
Let $n \in {\mathbb P}$, $n \geq 5$, and $i \in [3,n-1]$. Then 
\[
D_{n}^{\{ 0, \, i \}}(x) =
\prod_{j=4}^n (1+(-1)^{j-1} x^{\left\lfloor\frac{j}{2} \right\rfloor})^2.
\]
\end{con}

\begin{con}
\label{con01i}
Let $n \in {\mathbb P}$, $n \geq 5$, and $i \in [3,n-1]$. Then 
\[
D_{n}^{\{ 0,1, \, i \}} (x) =
\frac{1+x^2}{1-x^2} \; \prod_{j=4}^n (1+(-1)^{j-1} x^{\left\lfloor\frac{j}{2} \right\rfloor})^2.
\]
\end{con}

We have verified these conjectures for $n \leq 8$. Note that, by Proposition \ref{shd}, it
is enough to prove Conjectures \ref{con0i} and \ref{con01i} for $i=3$.

Note that, by Theorem \ref{Aquot}, Conjectures %\ref{con02}, 
\ref{con0i} and \ref{con01i} may
be formulated in the following equivalent way. For $n \in {\mathbb P}$ and $J \subseteq [n-1]$ 
let $S_{n}^{J}(x) \eqdef \sum_{\sigma \in S_{n}^{J}} (-1)^{\ell_A (\sigma )} x^{L_{A}(\sigma )}$.

\begin{con}
\label{con0ib}
Let $n \in {\mathbb P}$, $n \geq 5$, and $i \in [3,n-1]$. Then 
\[
D_{n}^{\{ 0, \, i \}} (x) = (S_n^{ \{ i \} } (x))^2.
\]
\end{con}

\begin{con}
\label{con01ib}
Let $n \in {\mathbb P}$, $n \geq 5$, and $i \in [3,n-1]$. Then 
\[
D_{n}^{\{ 0,1,i \}} (x) =
(1-x^4) \, (S_n^{ \{ 1, \, i \} } (x))^2.
\]
\end{con}

We feel that the presence of the factor 
$\prod_{j=4}^n (1+(-1)^{j-1} x^{\left\lfloor\frac{j}{2} \right\rfloor})^2$ in Conjectures 
\ref{con0i} and \ref{con01i} is not a coincidence. More generally, we feel that 
the following holds.

\begin{con}
Let $n \in {\mathbb P}$, $n \geq 3$, and $J \subseteq [0,n-1]$. Let $J_0, J_1, \ldots , J_s$
be the connected components of $J$ indexed as described before Theorem \ref{Bquot}. Then
there exists a polynomial $M_J (x) \in {\mathbb Z}[x]$ such that
\[ D_{n}^{J} (x) =
M_J (x) \; \prod_{j=2m+2}^{n} (1+(-1)^{j-1} x^{\left\lfloor \frac{j}{2} \right\rfloor })^2 ,
\]
where $m:=\sum_{i=0}^s \left\lfloor \frac{|J_{i}|+1}{2}\right\rfloor$. Furthermore,
$M_J(x)$ only depends on $( |J_0|, |J_1|, \ldots , |J_s| )$ and is a symmetric function of
$ |J_1|, \ldots , |J_s| $.
\end{con}

This conjecture has been verified for $n\leq 8$.

\end{document}